\documentclass[12pt]{amsart}
\usepackage[all]{xy}
\usepackage{pdfsync}
 \usepackage[colorlinks=false]{hyperref}
\usepackage{amssymb} 
\usepackage{amsmath} 
\usepackage{graphicx} 
\usepackage[english]{babel} 
\usepackage{epsfig}
\usepackage{multirow}
\usepackage{color}
\usepackage{enumerate}
\usepackage[T1]{fontenc}
\usepackage{aurical, pbsi,calligra}
\swapnumbers

\theoremstyle{definition}
\newtheorem{definition}{Definition}[section]

\newtheorem{example}[definition]{Example}

\theoremstyle{plain}
\newtheorem{lemma}[definition]{Lemma}

\newtheorem{theorem}[definition]{Theorem}
\usepackage[all]{xy}
\usepackage{pdfsync}
\usepackage{ytableau}
\usepackage{rotating}

\begin{document}

\title{Unary and binary Leibniz algebras}

\author{N.A. Ismailov}

\address{Astana IT University, Nur-Sultan, Kazakhstan and Institute of Mathematics and Mathematical Modeling, Almaty, Kazakhstan}

\email{nurlan.ismail@gmail.com}

\author{A.S. Dzhumadil'daev}

\address{Kazakh-British Technical University and Institute of Mathematics and Mathematical Modeling, Almaty, Kazakhstan}

\email{dzhuma@hotmail.com}

% 17A30 Algebras satisfying other identities
% 17A50 Free algebras

\thanks{The first author was supported by grant AP08052405 of MES RK}

 \maketitle

\begin{abstract} An algebra is said to be a unary Leibniz algebra if every one-generated subalgebra is a Leibniz algebra. An algebra is said to be a binary Leibniz algebra if every two-generated subalgebra is a Leibniz algebra. We give characterizations of unary and binary Leibniz algebras in terms of identities.   
\end{abstract}

\section{\label{nn}\ Introduction}

For a class of algebras $\mathcal{A}$, we let $\mathcal{A}_1$ denote the class of algebras such that every  one-generated algebra belongs to $\mathcal A.$ Similarly,   ${\mathcal A}_2$ is defined as the class of algebras such that every two-generated algebra belongs to ${\mathcal A}.$

An algebra $A$ is called  {\it unary } $\mathcal{A}$-algebra if $A\in \mathcal{A}_1.$ 
Similarly, $A$ is {\it binary} $\mathcal A$-algebra if $A\in \mathcal{A}_2.$ 

For example, let us consider $\mathcal{A}=\mathcal{A}s$, the class of associative algebras, that is, the class of algebras generated by the associative identity:
$$(a,b,c)=(ab)c-a(bc)=0.$$
By Artin's theorem \cite{ZSSS}, the class of binary associative algebras  coincides with the class of alternative algebras, i.e., ${\mathcal A}s_2$ is generated by the polynomial identities
$$(a,b,c)=-(b,a,c)=-(a,c,b).$$
It follows from Albert's theorem \cite{Albert} and \cite{Gainov} that over a field of characteristic zero, the class of unary associative algebras coincides with the class of power-associative algebras, i.e., $\mathcal{A}s_1$ is generated by the following identities:
$$(a,a,a)=0, \quad (aa,a,a)=0.$$

Another example concerns Lie algebras ${\mathcal A}=\mathcal{L}ie$, the class defined by the identities
$$ab+ba=0,$$
$$Jac(a,b,c)=0,$$
where $Jac(a,b,c)=(ab)c+(bc)a+(ca)b$ is the jacobian of elements $a,b,c$.
In this case, the class of binary Lie algebras $\mathcal{L}ie_2$ over a field of characteristic different from two  is defined by the following polynomial identities \cite{Gainov}:
$$ab+ba=0,$$
$$Jac(a,b,ab)=0.$$
We also note that the class of unary Lie algebras coincides with the class of anticommutative algebras.

In this paper we study the classes of unary and binary Leibniz algebras. Recall that (left)-Leibniz algebras are defined by the following identity:
$$(ab)c=a(bc)-b(ac).$$
Before formulating our results, let us define a notation. 

Let ${\bf K}$ be a field of characteristic not 2 and $A$ be an algebra over ${\bf K}$. For any $a,b,c\in A$ define $$\langle a,b,c\rangle=(ab)c-a(bc)+b(ac).$$ Then, clearly an algebra with the identity $\langle a,b,c\rangle=0$ becomes a Leibniz algebra. In the case of anticommutative  algebras, we recall that $$\langle a,b,c\rangle=Jac(a,b,c).$$
Our first result provides the characterization of unary algebras.
\begin{theorem}\label{first main theorem}
An algebra A over ${\bf K}$ is unary Leibniz if and only if it satisfies the following identities:
\begin{equation}\label{equn1}
\langle a,a,a\rangle=0,
\end{equation}
\begin{equation}\label{equn2}
\langle aa,a,a\rangle=0.
\end{equation}
\end{theorem}

The next theorem characterizes the binary Leibniz algebras.

\begin{theorem}\label{second main theorem}
An algebra A over ${\bf K}$ is binary Leibniz if and only if it satisfies the following identities:
\begin{equation}\label{eq1}
\langle a,a,b\rangle=0,
\end{equation}
\begin{equation}\label{eq2}
\langle a,b,a\rangle=0,
\end{equation}
\begin{equation}\label{eq3}
\langle a, b, ab\rangle=0.
\end{equation}
\end{theorem} 

Identities (\ref{equn1}) and (\ref{equn2}) are equivalent, respectively, to the following identities:
$$(aa)a=0,$$
$$(aa)(aa)=0.$$
Expanded forms of identities (\ref{eq1}), (\ref{eq2})  and (\ref{eq3}) are given below:
$$(aa)b=0,$$
$$b(aa)=-(a,b,a),$$
$$(ab)(ab)=a(b(ab))-b(a(ab)).$$

It is well known that every Lie algebra is a Malcev algebra and every Malcev algebra is a binary Lie algebra. However, there exists a binary Lie algebra which is not Malcev, and there is a Malcev algebra which is not Lie \cite{Kuzmin}. Since Leibniz algebras are noncommutative generalizations of Lie algebras, any binary Lie algebra is binary Leibniz. By Theorem \ref{second main theorem} it follows that any Leibniz algebra is binary Leibniz and every binary Leibniz algebra is unary Leibniz. Moreover, Theorem \ref{first main theorem} gives that any unary Lie algebra is unary Leibniz. The picture of inclusions looks as follows: 

$$
\begin{array}{c c c c c c c}
\mathcal{L}ie  & \subset  & \mathcal{M}alc & \subset & \mathcal{L}ie_2  &\subset & \mathcal{L}ie_1\\
\rotatebox{-90}{$\subset$}& &                  &&\rotatebox{-90}{$\subset$}& & \rotatebox{-90}{$\subset$}\\
\mathcal{L}eib& \subset &    & \subset &\mathcal{L}eib_2 & \subset &  \mathcal{L}eib_1\\
\end{array}
$$

We observe that the inclusions above are strict. In what follows, we give examples that confirm these facts.

\begin{example}\label{ex1} 
({\it Non-Leibniz, but  binary Leibniz algebra}). Let $A=\{e_1,e_2,e_3,e_4\}$ be a four dimensional anticommutaive algebra with the multiplication table
$$e_1e_2=e_3, \quad  e_1e_4=e_1, \quad e_2e_4=e_2, \quad e_3e_4=-e_3$$
and zero products are omitted. In this case, $A$ becomes a non-Lie Malcev algebra \cite{Kuzmin}, hence $A$ is binary Leibniz. However, $A$ is not Leibniz as $\langle e_1,e_2,e_4\rangle=-3e_3\neq 0.$ 
\end{example} 
\begin{example}\label{ex2}  
 ({\it Non-binary  Lie, but binary Leibniz algebra}). Let $B=\{e_1,e_2\}$ be a two dimensional algebra with the multiplication table
$$e_1^2=e_2$$
and zero products are omitted. Since it is not anticommutative, $B$ is not binary Lie. On the other hand, $B$ is a Leibniz algebra which makes $B$ into a binary Leibniz algebra.
\end{example}
\begin{example} \label{ex3} 
({\it Non-binary Lie, non-Leibniz, but binary Leibniz algebra}). Direct sum of two algebras defined above $A\oplus B$ gives an example of a binary Leibniz algebra that is neither Leibniz nor binary Lie. 
\end{example} 
\begin{example} \label{ex4} 
({\it Non-binary, non-unary Lie, but unary Leibniz algebra}). Let $C=\{e_1,e_2,e_3,e_4\}$ be a four dimensional   non-anticommutative algebra with the multiplication table
$$e_1e_1=e_4e_4=e_2, \quad e_1e_2=e_3$$
and zero products are omitted. Then  $C$ is unary Leibniz, but it is not binary Leibniz as $\langle e_1+e_4,e_4,e_1+e_4\rangle=-e_3\neq 0.$ 
\end{example}
\begin{example}\label{ex5}  
({\it Non-binary Lie, but  unary Lie algebra}). Let $D=\{e_1,e_2,e_3\}$ be a three dimensional anticommutative algebra with the multiplication table
$$e_1e_2=e_3,  \quad e_1e_3=e_1,  \quad e_2e_3=e_2$$
and zero products are omitted. $D$ is not a binary Lie algebra, since $Jac(e_1,e_2,e_1e_2)=-2e_3\neq 0$.
\end{example}

In the proof of Theorems \ref{first main theorem} and \ref{second main theorem} we use 
a linearization method. Let $a,b,f\in A$ and $\Delta_a^1(b)$ be a partial linearization mapping of $A.$ We refer to \cite[Chapter 1]{ZSSS} for properties of partial linearization mappings. 

%Recall its definition. We suppose that $f$ is a linear combination of monomials on some variable  one of which is $a.$ We change  in each monomial the variable $a$ by $b$ consecutively from left to right. The polynomial $(f)\bigtriangleup_a^1(b) $ is a linear combination of such monomials. For example, if $f=f(a,b,c,d)=2(ab)(ac)- 3(d(aa))b,$ then
%$$(f)\bigtriangleup_a^1(b) =2(bb)(ac)+2(ab)(bc)-3(d(ba))b-3(d(ab))b.$$

In fact we write $\Delta(f)$  instead of $(f)\Delta_a^1(b)$ and call such an operation {\it differential substitution} $a\rightarrow b$ in $f$. The exact meaning of $\Delta(f)$ will be clear from the context. We let $deg\,(f)$ denote the degree of an element $f.$ 

In the binary Lie case, for any $u$ there is a relation $\Delta(u)=uy$ when $x\rightarrow xy$ in $u$ \cite{Gainov}. However, in general,  this relation is not valid in the case of binary Leibniz. Therefore, we give additional binary Leibniz arguments to show that 
$$(\Delta(u)+uy)v=2(uy)v.$$
This observation is essential in the proof of Theorem \ref{second main theorem}.

\section{\label{nn}\ Proof of Theorem \ref{first main theorem}}

Let $A$ be an algebra with identities (\ref{equn1}) and (\ref{equn2}). For some element $x\in A$ set $e_n=\underbrace{x(\cdots(xx)\cdots)}_{n\text{-times}}.$  It is clear that $e_1e_n=e_{n+1}.$

The next statement is crucial in the proof of Theorem \ref{first main theorem}.

\begin{lemma}\label{helpful lemma2}
An algebra with identities (\ref{equn1}) and (\ref{equn2}) satisfies the identities $e_ke_l=0$ for any $k>1$, $l>0$. 
\end{lemma}
\begin{proof}
We prove our statement by induction on degree $n$ of $e_ke_l$, where $n=k+l$. 

If $n=3$, then from identity (\ref{equn1}) we have $e_2e_1=0.$ 

Assume $n=4$. The relation $e_2e_2=0$ follows from identity (\ref{equn2}). Let us prove $e_3e_1=0$. Using   the differential substitution $e_1\rightarrow e_2$ in $e_2e_1=0$, we obtain
$$0=\Delta(e_2e_1)=(e_2e_1)e_1+(e_1e_2)e_1+e_2e_2=(e_1e_2)e_1=e_3e_1.$$

Suppose that $e_ke_l=0$ if $deg\,(e_ke_l)\le n$ and $k>1, l>0.$ 
We prove our statement for $n+1$. 

The differential substitution $e_1\rightarrow e_2$ in $e_k$ 
gives us
$$\Delta(e_k)=e_2(e_{k-1})+e_1(e_2e_{k-2})+\ldots+e_1(e_1(e_1(\cdots(e_2e_1)\cdots)))+e_{k+1}.$$
By the induction hypothesis we have $$e_2e_{k-1}=0,\quad e_1(e_2e_{k-2})=0, \ldots, e_1(\cdots (e_1(e_2e_1))\cdots)=0.$$ Therefore, $\Delta(e_k)=e_{k+1}$ for $k<n$. Hence
$$\Delta(e_ke_l)=e_{k+1}e_l+e_ke_{l+1}=0$$
for all $k>1$ and $l>0$ with $k+l=n.$
So, we derive 
\begin{equation}\label{equ1}e_ne_1=-e_{n-1}e_2=\cdots=(-1)^ne_2e_{n-1}.\end{equation}

Let us make the  differential substitution $e_1\rightarrow e_{n-1}$ in $e_2e_1=0$.  By the induction hypothesis we get
$$\Delta(e_2e_1)=(e_{n-1}e_1)e_1+(e_1e_{n-1})e_1+e_2e_{n-1}=e_ne_1+e_2e_{n-1}=0,$$
that is, 
 \begin{equation}\label{equ2}e_ne_1=-e_2e_{n-1}.\end{equation}
Now we use the differential substitution $e_1\rightarrow e_{n-2}$ in $e_2e_2=0$.
By the induction hypothesis we have
$$\Delta(e_2e_2)=(e_{n-2}e_1)e_2+(e_1e_{n-2})e_2+e_2(e_{n-2}e_1)+e_2(e_1e_{n-2})=$$
$$e_{n-1}e_2+e_2e_{n-1}=0.$$ 
So,
\begin{equation}\label{equ3}e_{n-1}e_2=-e_2e_{n-1}.\end{equation} 
   
If $n$ is even, then (\ref{equ1}) and (\ref{equ2}) imply $e_ne_1=0$.
If $n$ is odd, then by (\ref{equ1}) and (\ref{equ3}) we obtain $e_{n-1}e_2=0.$
In both cases by (\ref{equ1}) we see that
 $e_ke_l=0$ for any $k>1,l>0$ with $k+l=n+1$.  
\end{proof}

{\it Proof of Theorem \ref{first main theorem}.}  Since the identities (\ref{equn1}) and (\ref{equn2})  depend only on one variable, any unary Leibniz algebra satisfies these identities  over a field of any characteristic. 

Let $x\in A$ and $alg\langle x\rangle$ be a one-generated subalgebra of $A$ on $x$. We show that, for any $u,v,w\in alg\langle x\rangle$,
$$\langle u,v,w\rangle=0.$$ 
By Lemma  \ref{helpful lemma2} it is enough to check it for elements   $u=e_k$, $v=e_l$ and $w=e_m.$

If $k=l=1$, then
$$\langle e_1,e_1,e_m\rangle=(e_1e_1)e_m-e_1(e_1e_m)+e_1(e_1e_m)=e_2e_m=0.$$
If $k=1$ and $l>1$, then
$$\langle e_1,e_l,e_m\rangle=(e_1e_l)e_m-e_1(e_le_m)+e_l(e_1e_m)=e_{l+1}e_m+e_le_{m+1}=0.$$
If $k>1$ and $l=1$, then
$$\langle e_k,e_1,e_m\rangle=(e_ke_1)e_m-e_k(e_1e_m)+e_1(e_ke_m)=-e_ke_{m+1}=0.$$
If $k>1$ and $l>1$, then
$$\langle e_k,e_l,e_m\rangle=(e_ke_l)e_m-e_k(e_le_m)+e_l(e_ke_m)=0. \quad \square $$

\section{\label{nn}\ Proof of Theorem \ref{second main theorem}}

Let $A$ be an algebra with identities (\ref{eq1}), (\ref{eq2}) and (\ref{eq3}).

The complete linearizations of these identities are respectively equivalent to 
 \begin{equation}\label{eq7}
\langle a,b,c\rangle+\langle b,a,c\rangle=0,
\end{equation}
\begin{equation}\label{eq8}
\langle a,b,c\rangle+\langle c,b,a\rangle=0,
\end{equation}
\begin{equation}\label{eq9}
\langle a,b,cd\rangle+\langle a,d,cb\rangle+\langle c,b,ad\rangle+\langle c,d,ab\rangle=0.
\end{equation}

\begin{lemma}\label{helpful lemma}
For all $a,b,c,d\in A$, we have
 
(a) $\langle a,b,c\rangle$ is skew-symmetric in $a,b,c$.

(b) $\langle ab,c,d\rangle+\langle ba,c,d\rangle=0.$
\end{lemma}
\begin{proof}

(a) It follows from the identities (\ref{eq7}) and (\ref{eq8}).

(b) The identity (\ref{eq7}) implies $(ab)c=-(ba)c.$ 
Therefore, we have
 $$\langle ab,c,d\rangle=((ab)c)d-(ab)(cd)+c((ab)d)=$$$$-((ba)c)d+(ba)(cd)-c((ba)d)=-\langle ba,c,d\rangle.$$ 
\end{proof}
{\it Proof of Theorem \ref{second main theorem}.} Since the identities (\ref{eq1}), (\ref{eq2}) and (\ref{eq3}) depend only on two variables, any binary Leibniz algebra satisfies these identities over a field of any characteristic. 

Let $x,y\in A$ and $alg\langle x,y\rangle$ be a two-generated subalgebra of $A$ on $x,y$. We show that 
$$\langle u,v,w\rangle=0$$ for any $u,v,w\in alg\langle x,y\rangle$. We prove this statement by induction on $deg\,(\langle u,v,w\rangle)$. If $deg\,(\langle u,v,w\rangle)$ is 3 or 4, then by identities (\ref{eq1}), (\ref{eq2}) and (\ref{eq3}) our statement is obvious. 

Assume that $\langle u,v,w\rangle=0$ if  $deg\,(\langle u,v,w\rangle)$ is not greater than $n$. Suppose $deg\,(\langle u,v,w\rangle)=n+1$. If $deg\,(w)\geq 2$, then by the differential substitution $x\rightarrow w$ in $\langle u,v,x\rangle=0$ we obtain 
$$\Delta(\langle u,v,x\rangle)=\langle \Delta(u),v,x\rangle+\langle u,\Delta(v),x\rangle+\langle u,v,w\rangle=0$$
So, once we prove that all triples of the form  $\langle u,v,x\rangle$ of degree $n+1$ are zero in $A$, we can get $\langle u,v,w\rangle=0$. 

Now we show that, in fact, $\langle u,v,x\rangle$ of degree $n+1$ is a linear combination of triples of the form $\langle u,x,y\rangle$. Suppose $deg\,(v)\geq 2$ and consider the differential substitution $y\rightarrow v$ in $\langle u,y,x\rangle=0$. Then we have $$\Delta(\langle u,y,x\rangle)=\langle \Delta(u),y,x\rangle+\langle u,v,x\rangle+\langle u,y,\Delta(x)\rangle=0.$$  It gives the following relation:
 $$\langle u,v,x\rangle=-\langle \Delta(u),y,x\rangle.$$
To complete the proof it is left to show that all triples of the form $\langle u,x,y\rangle$ are zero in $A$.  

Suppose $deg\,(\langle u,x,y\rangle)=n$. The differential substitution $x\rightarrow xy$ in $\langle u,x,y\rangle$ gives us the following relation: 
$$\Delta(\langle u,x,y\rangle)=\langle \Delta(u),x,y\rangle+\langle u,\Delta(x),y\rangle+\langle u,x,\Delta(y)\rangle=0.$$
It follows 
\begin{equation}\label{eq10}
\langle \Delta(u),x,y\rangle+\langle u,xy,y\rangle=0.
\end{equation}
By the differential substitution $x\rightarrow u$ in $\langle xy,x,y\rangle$ we obtain  
$$\Delta(\langle xy,x,y\rangle)=\langle uy,x,y\rangle+\langle xy,\Delta(x),y\rangle+\langle xy,x,\Delta(y)\rangle=0.$$
It follows 
\begin{equation}\label{eq11}
\langle uy,x,y\rangle+\langle xy,u,y\rangle=0.
\end{equation}
Take the sum of (\ref{eq10}) and (\ref{eq11}).  By  Part (a) of Lemma \ref{helpful lemma} we have
\begin{equation}\label{eq12}
\langle \Delta(u)+uy,x,y\rangle=0.
\end{equation}
Now we show that for some $v\in A$ in the differential substitution $x\rightarrow xy$ in $u$ of degree $n-2$, the relation
\begin{equation}\label{eq13}
(\Delta(u)+uy)v=2(uy)v
\end{equation} holds.
We prove it by induction on degree of $u$. If $u=xy$, then it is obvious. Suppose $u=yx$. Then by (\ref{eq7}) 
$$(\Delta(yx)+(yx)y)v=$$$$(y(xy)+(yx)y)v=((yx)y+(yx)y)v=2((yx)y)v.$$
Assume that it is true for all $u$ whose degree is less than $n-2$. Suppose $deg\,(u)=n-2$ and $u=u_1u_2$. Since $deg\,((u_1u_2)y)=n-1$, by the first induction hypothesis we can write $(u_1u_2)y=u_1(u_2y)-u_2(u_1y)$. Then by (\ref{eq7})  and the induction hypothesis we have
$$(\Delta(u)+uy)v=$$$$(\Delta(u_1u_2)+(u_1u_2)y)v=(\Delta(u_1)u_2+u_1\Delta(u_2)+(u_1u_2)y)v=$$
$$(\Delta(u_1)u_2+u_1\Delta(u_2)+u_1(u_2y)-u_2(u_1y))v=$$$$(\Delta(u_1)u_2-\Delta(u_2)u_1-(u_2y)u_1+(u_1y)u_2)v=$$$$ ((\Delta(u_1)+u_1y)u_2-(\Delta(u_2)+u_2y)u_1)v=(2(u_1y)u_2-2(u_2y)u_1)v=$$$$2((u_1y)u_2-(u_2y)u_1)v=2(u_1(u_2y)-u_2(u_1y))v=2((u_1u_2)y)v=$$
$$2(uy)v.$$

So by (\ref{eq12}), (\ref{eq13}) and by Part (b) of Lemma \ref{helpful lemma} we can derive $$0=\langle \Delta(u)+uy,x,y\rangle=\langle 2uy,x,y\rangle=2\langle uy,x,y\rangle=-2\langle yu,x,y\rangle.$$
Since {\bf K} is the field of characteristic not 2, we have 
\begin{equation}\label{eq14}
\langle yu,x,y\rangle=0.
\end{equation}
In a similar way, one can show that 
\begin{equation}\label{eq15}
\langle xv,x,y\rangle=0
\end{equation}
for any $v$ of degree $n-2$.
In \cite{LodayPirashvili} it was proved that set of right-normed bracketing elements on a set forms a base of free (left)-Leibniz algebra on a given set. Therefore, any element $w\in A$ of degree $n-1$ can be expressed as a linear combination of elements of the form $yu$ and $xv$. So by (\ref{eq14}) and (\ref{eq15}), for any $w\in A$, we have $$\langle w,x,y\rangle=0.$$
and it completes our proof. $\square$ 

\section{\label{nn}\ Remarks and open questions}

{\bf 1.} One notices that between $\mathcal{L}eib$ and $\mathcal{L}eib_2$ there is an empty place. 

$$
\begin{array}{c c c c c c c}
\mathcal{L}ie  & \subset  & \mathcal{M}alc & \subset & \mathcal{L}ie_2  &\subset & \mathcal{L}ie_1\\
\rotatebox{-90}{$\subset$}& &  \rotatebox{-90}{$\subset$}                &&\rotatebox{-90}{$\subset$}& & \rotatebox{-90}{$\subset$}\\
\mathcal{L}eib& \subset &  ?  & \subset &\mathcal{L}eib_2 & \subset &  \mathcal{L}eib_1\\
\end{array}
$$

It would be interesting to find a Leibniz analogue of Malcev algebras (if exists)  that makes our picture complete.

{\bf 2.} In \cite{Bremner-Peresi-Sanchez} Bremner, Pereci and  S\'{a}nchez-Ortega defined Malcev dialgebras over a field of characteristic not 2. A dialgebra is said to be {\it (left)-Malcev} dialgebra if it satifies the following identities:
$$(ab)c+(ba)c=0,$$
$$a(b(cd))-b(c(ad))-c((ab)d)-(ac)(bd)-(a(bc))d=0.$$
The first identity is called {\it left anticommutativity} and the second identity is called {\it di-(left)-Malcev} identity. Let $di$-$\mathcal{M}alc$ be the class of Malcev dialgebras. Since Malcev dialgebra is a "noncommutative" version of Malcev algebras, we have
\begin{center}
$\mathcal{M}alc\subset di$-$\mathcal{M}alc$.
\end{center}
It is easy to check that
\begin{center}
$\mathcal{L}eib\subset di$-$\mathcal{M}alc$\end{center} and \begin{center}$di$-$\mathcal{M}alc\subset di$-$\mathcal{L}eib_1$.  
\end{center}
However, \begin{center}$di$-$\mathcal{M}alc\not\subset\mathcal{L}eib_2$.\end{center} In order to verify this let us consider the algebra $C$ from Example \ref{ex4}, that is, the algebra with the multiplication table 
$$e_1e_1=e_4e_4=e_2, \quad e_1e_2=e_3$$
and zero products are omitted. We note that $e_2$ and $e_3$ are left-central elements of $C$, that is, 
$$e_2 a=0 \quad \text{and} \quad e_3 a=0$$ for any $a\in C$, and $C^2$ is spanned by $e_2$ and $e_3$. One can check that $C$ is nilpotent algebra with nilpotency index 4. Then $C$ satisfies left anticommutativity and di-Malcev identity and consequently, $C$ is a Malcev dialgebra. From Example \ref{ex4} we know that $C$ is not a binary Leibniz algebra. Therefore, any Malcev dialgebra is not a binary Leibniz algebra. Hence, the class of Malcev dialgebras can not fill the empty space in the picture.   
   
{\bf 3.} 
Another interesting class of algebras that contains all Leibniz and Malcev algebras is the class of binary Lie dialgebras. Let $di$-$\mathcal{L}ie_2$ be the class of binary Lie dialgebras. In order to derive defining identities of $di$-$\mathcal{L}ie_2$, one can use the general algorithm described in \cite{Kolesnikov}. Then a dialgebra is said to be a {\it binary Lie} dialgebra if it satisfies the following identity: 
$$a(b(cd))+a(c(bd))-((ac)b)d-((ab)c)d$$
$$-b(a(cd))-c(a(bd))+b((ca)d)+c((ba)d)=0.$$    
We rewrite the di-Malcev identity as follows:
$$c(a(bd))-a(b(cd))-b((ca)d)-(cb)(ad)-(c(ab))d=0$$  
and
$$b(a(cd))-a(c(bd))-c((ba)d)-(bc)(ad)-(b(ac))d=0.$$  
To derive the binary Lie dialgebra identity we apply the left anticommutativity identity to the sum of the last two Malcev identities. Then
\begin{center}$di$-$\mathcal{M}alc\subset di$-$\mathcal{L}ie_2$\end{center}
and therefore
\begin{center}$\mathcal{L}eib\subset di$-$\mathcal{L}ie_2$\end{center} and  \begin{center}$\mathcal{M}alc\subset di$-$\mathcal{L}ie_2$.\end{center} Since 
\begin{center}$di$-$\mathcal{M}alc\not\subset\mathcal{L}eib_2$,\end{center}
we obtain \begin{center}$di$-$\mathcal{L}ie_2\not\subset\mathcal{L}eib_2$.\end{center}  
Hence, the class of binary Lie dialgebras can not fill the empty space in the picture.

{\bf 4.} 
The class of Leibniz algebras has no simple non-Lie algebras, but the class of binary Lie algebras has one simple non-Lie algebra, which is in fact Malcev algebra \cite{Grishkov}. Hence the class of binary Leibniz  algebras contains  simple non-Lie algebras. It would also be interesting if one finds non-binary Lie, but simple binary Leibniz algebra (if exists).

{\bf Acknowledgments.} The authors are grateful to Professor P.S. Kolesnikov for his discussion and for essential comments which have been very helpful in improving the manuscript.

\end{document}